\documentclass[12pt,reqno]{amsproc}
\usepackage[all]{xy}
\usepackage{amsmath,graphicx,amssymb,amsthm, mathrsfs, amsfonts,dsfont}
\usepackage{rotating}
\usepackage{CJKutf8}
\allowdisplaybreaks
\def\A{\mathcal A}
\def\B{\mathcal B}

\def\B{\mathcal B}

\def\M{\mathcal M}
\def\N{\mathcal N}

\def\R{\mathcal R}

\def\CCC{\mathbb C}
\def\NNN{\mathbb N}

\def\RRR{\mathbb R}

\def\amslatex{$\mathcal{A}\kern-.1667em\lower.5ex\hbox{$\mathcal{M}$}\kern-.125em\mathcal{S}$-\LaTeX}

\newtheorem{set}{set}[section]

\newtheorem{Lemma}[set]{Lemma}

\newtheorem{Theorem}[set]{Theorem}
\newcommand{\define}{\mathrel{\hbox{$\equiv$\hskip -.90em \lower .47ex \hbox{$\leftharpoondown$}}}}
\newcommand{\enifed}{\mathrel{\hbox{$\equiv$\hskip -.90em \lower .47ex \hbox{$\rightharpoondown$}}}}

\numberwithin{equation}{section}
\makeatletter
\newcommand{\LeftEqNo}{\let\veqno\@@leqno}

\makeatother

\begin{document}
\begin{CJK}{UTF8}{<font>}

\title{On Diximier's averaging theorem for operators in type ${\rm II}_1$ factors}

\author{Shilin Wen}
\address{School of Mathematics and Information Sciences, China West Normal University, Nanchong, Sichuan, 637000, China}
\email{shilinwen127@hotmail.com}
\thanks{Shilin Wen was partly supported by NSFC(Grant No.12001437) and  the Fundamental Research Funds of the China West Normal University (21E027)}

\author{Junsheng Fang}
\address{School of Mathematical Sciences, Hebei Normal University, Shijiazhuang, Hebei, 050024, China}
\email{jfang@hebtu.edu.cn}
\thanks{Junsheng Fang was partly supported by NSFC(Grant No.12071109) and a Start-up funding of Hebei Normal University.}

\author{Zhaolin Yao}
\address{School of Mathematical Sciences, Hebei Normal University, Shijiazhuang, Hebei, 050024, China}
\email{zyao@hebtu.edu.cn}
\thanks{Zhaolin Yao was partly supported by Science Foundation of Hebei Normal University(Grant No.13113115)}
\date{}
\maketitle
\begin{abstract}
Let $\M$ be a type ${\rm II_1}$ factor and let $\tau$ be the faithful normal tracial state on $\M$. In this paper, we prove that given finite elements $X_1,\cdots X_n \in \M$, there is a finite decomposition
of the identity into $N \in \NNN$ mutually orthogonal nonzero projections $E_j\in\M$, $I=\sum_{j=1}^NE_j$, such that $E_jX_iE_j=\tau(X_i) E_j$ for all $j=1,\cdots,N$ and $i=1,\cdots,n$. Equivalently,  there is a
unitary operator $U \in \M$ such that
$\frac{1}{N}\sum_{j=0}^{N-1}{U^*}^jX_iU^j=\tau(X_i)I$ for $i=1,\cdots,n$. This result is a stronger version of Dixmier's averaging theorem for type ${\rm II}_1$ factors. As the first application, we show that all elements of trace zero in a type ${\rm II}_1$ factor are single commutators and any self-adjoint elements of trace zero are single self-commutators. This result answers affirmatively Question 1.1 in~\cite{DS}. As the second application, we prove that any self-adjoint element in a type ${\rm II}_1$ factor can be written a linear combination of 4 projections. This result answers affirmatively Question 6(2) in \cite{GP}. As the third application, we show that if $(\mathcal{M},\tau)$ is a finite factor, $X \in \mathcal{M}$, then there exists a normal operator  $N \in \mathcal{M}$  and a nilpotent operator $K$ such that $X= N+ K$. This result answers affirmatively Question 1.1 in \cite{DNZ}.

\end{abstract}



\vskip1.0cm

\section{Introduction}
Let $\M$ be a type ${\rm II_1}$ factor and let $\tau$ be the faithful normal tracial state on $\M$. In~\cite{CFY2}, the following result is proved. If $X=X^*\in \M$, then there is a decomposition of the identity into $N\in \NNN$ mutually orthogonal nonzero projections $E_j$, $I=\sum_{j=1}^NE_j$, for which $E_jXE_j=\tau(X) E_j$ for all $j=1,...,N$, equivalently, there is a unitary operator $U \in \M$ with
$\frac{1}{N}\sum_{j=0}^{N-1}{U^*}^jXU^j=\tau(X)I$. A natural question arises: can we remove the selfadjoint assumption on $X$. In~\cite{WC}, the authors proved that for an arbitrary operator $X$ in the ultrapower algebra of the hyperfinite type ${\rm II}_1$ factor $\R^w$, there is a decomposition of the identity into $N\in \NNN$ mutually orthogonal nonzero projections $E_j\in \R^w$, $I=\sum_{j=1}^NE_j$, for which $E_jXE_j=\tau(X) E_j$ for all $j=1,...,N$. The main result of this paper generalizes the above results to arbitrary operators in any type ${\rm II}_1$ factors. Precisely, we prove the following result.

\begin{Theorem}[Main Theorem]
Let $(\mathcal{M},\tau)$ be a type ${\rm II}_1$ factor, $X_1,\cdots, X_n \in \mathcal{M}$. Then there is a family of finitely many mutually orthogonal non-zero projections $\{E_i \}_{i=1} ^{N}$ in $\mathcal{M}$ such that $\sum\limits_{i=1} ^{N}E_i =I$ and $E_iX_jE_i=\tau (X_j)E_i$ for $1\leq i\leq N$ and $1\leq j\leq n$. That is, $X_j$ can be written as follows
\[  X_j=
  \bordermatrix	
  {&E_1&E_2&\cdots&E_N \cr
  	E_1 &\tau(X_j) & (X_j)_{12} & \cdots & (X_j)_{1N} \cr
  	E_2 &(X_j)_{21}& \tau(X_j) & \cdots & (X_j)_{2N} \cr
  	\vdots & 	\vdots & \vdots & \ddots & \vdots\cr
  	E_N & (X_j)_{N1} & (X_j)_{N2} & \cdots &\tau(X_j) \cr}.
  \]
Equivalently, there is a unitary operator $U\in \M$ such that \[\frac{1}{N} \sum_{i=0}^{N-1}(U^i)^\ast X_j U^i=\tau(X_j)I\] for $1\leq j\leq n$.
\end{Theorem}

The main theorem has several interesting applications.\\

Application on the commutators in type ${\rm II}_1$ factors.

An operator $A$ in a von Neumann algebra $\M$ is a commutator in $\M$ if there exist operators $B$ and $C$ in $\M$ such that $A=BC-CB$.
A self-commutator is a self-adjoint operator in $\M$ which is of the form $X^*X-XX^*$ for some operator $X \in \M$.
The problem of specifying which operators in $\M$ are commutators in $\M$ has been attacked by several authors and has been solved in certain special cases. K. Shoda in \cite{Shoda} showed that if $\M$ is a type ${\rm I}_n$ factor, then a matrix has trace zero if and only if it is a commutator in $\M$. A. Brown and C. Pearcy showed that if $\M$ is a factor of type ${\rm I}_{\infty}$, then non-commutators in $\M$ are exactly the operators that are congruent to a non-zero scalar modulo the ideal of compact operators \cite{BP2}. A. Brown and C. Pearcy in \cite{BP1} showed that if $\M$ is a factor of type ${\rm III}$ acting on a separable Hilbert space, then the commutators in $\M$ consist exactly of the non-scalar operators together with the operator zero. In \cite{BP}, A. Brown, C. Pearcy and D. Topping showed that every operator in the strong radical of $\M$ is a commutator in an arbitrary properly infinite von Neumann algebra $\M$.

However, The theory is far from complete and one of the most intractable of unsolved problems is that of determining the commutators in a finite von Neumann algebra. A commutator in a finite von Neumann algebra must have central trace zero, and is not unreasonable to hope that the commutators in such an algebra are exactly the operators with central trace zero.

The case of type ${\rm II}_1$ factors remains open. Some partial results are known. Fack and de la Harpe in \cite{FP} showed that every element of trace zero is a sum of ten commutators (and with control of the norms of the elements). S. Goldstein and A. Paszkiewicz improved the number(see \cite{GP}). They proved that any self-adjoint operator of central trace zero in a von Neumann algebra of type ${\rm II}_1$ is a sum of four commutators. The best previous estimates belonged to Marcoux in \cite{Marcoux} which showed that every element of trace zero in a type ${\rm II}_1$ factor is a sum of two commutators, and every self-adjoint element of trace zero is a sum of four or fewer self-commutators. Pearcy and Topping, in \cite{PT}, showed that in the type ${\rm II}_1$ factors of Wright every self-adjoint element of trace zero is a single commutator, and hence, every trace zero element is a sum of two commutators. Recently, K. Dykema and A. Skripka showed that all nilpotent elements are single commutators and every normal element with trace zero and purely atomic distribution is a single commutator(see \cite{DS}).

Using our main theorem, we show that all elements of trace zero in a ${\rm II}_1$ factor are single commutators and any self-adjoint elements of trace zero are single self-commutators. This result solves the problem of commutators in type ${\rm II}_1$ factors~\cite{DS}.\\

Application on the linear combination of projections in type ${\rm II}_1$ factors.

The linear combination of projections in von Neumann algebra factors is an active topic. For example, any self-adjoint operator acting in a finite dimensional Hilbert space can be written as a linear combination of a finite number of projections. A natural question is whether the number of projections does not depend on the dimension of Hilbert space. If the Hilbert space is infinite, is that possible? If the elements of algebra can be written as linear combinations of projections, a question naturally arises: what is the minimum number of projections?

The first positive results were obtained by Fillmore (see \cite{Fillmore}). He proved that any bounded operator acting on a separable infinite-dimensional Hilbert space can be written as a linear combination of 257 projections. Further, Fillmore was able to get down the number of projections to 9 in \cite{Fillmore69}. Pearcy and Topping in \cite{PT67} reduced it to 8, then A. Paszkiewicz to 6 (see \cite{Paszk80}). Using the techniques in \cite{Paszk80} , Matsumoto proved in 1984 that 5 is enough. Also in 1984, Nakamura \cite{Nakamura} showed that any self-adjoint operator can be written as a linear combination of 4 projections.

In \cite{GP},\cite{GP20} and \cite{GP22}, S. Goldstein and A. Paszkiewicz showed that each self-adjoint operator in type ${\rm II}_1$ factors is a linear combination of 12 projections and there exists a self-adjoint operator in any type ${\rm II}_1$ algebra that can not be written as a linear combination of 3 projections; Also for type ${\rm II}_{\infty}$ factors, any self-adjoint operator can be written as a linear combination of 4 projections and there exists a self-adjoint operator that can not be written as a linear combination of 3 projections; For countable decomposable type ${\rm III}$ factors, each self-adjoint element is a linear combination of 3 projections.

As the second application, we prove that any self-adjoint element in type ${\rm II}_1$ factor can be written a linear combination of 4 projections. This result answers affirmatively Question 6(2) in \cite{GP}.\\

Application on the generalization of Schur Theorem.

We state the famous theorem of Schur. For every matrix $T \in M_n(\CCC)$, there exists a unitary matrix $U\in M_n(\CCC)$ such that $U^{-1}TU$ is an upper triangular matrix.
The diagonal entries of $U^{-1}TU$ are the eigenvalues of $T$, repeated up to multiplicity, and $U$ can be chosen so that they appear in any order. Hence each ordering of the
spectrum of $T$ gives a decomposition $T=N+Q$, where $N$ is normal and $Q$ is nilpotent.

Dykema, Sukochev and Zanin in \cite{DSZ} used a Peano curve covering the support of the Brown measure of an operator $T$ in a diffuse, finite von Neumann algebra to give an ordering to the support of the Brown measure, and create a decomposition $T=N+Q$, where $N$ is normal and $Q$ is s.o.t.-quasinilpotent. This constructions generalize the Schur upper triangular form of an $n\times n$ matrix. The normal part $N$ is constructed as the conditional expectation of $T$ onto an abelian algebra generated by an increasing net of Haagerup-Schultz projections of $T$. It is natural to ask the question: under what circumstances is the s.o.t.-quasinilpotent operator $Q$ actually quasinilpotent.

As the third application, we show that if $(\mathcal{M},\tau)$ be a finite factor, $X \in \mathcal{M}$, then there exists a normal operator  $N \in \mathcal{M}$  and a nilpotent operator $K$ such that $X= N+ K$. This result answers affirmatively Question 1.1 in \cite{DNZ}.

The organization of this paper is as follows. Section 2 below is devoted to some results which will be very useful in section 3 and section 4. In Section 3, we prove the main theorem. In Section 4, we give three applications of the main theorem.

\section{Preliminaries}
\begin{Theorem}\label{thm: adjoint operator}{\rm (Theorem 1.1 in \cite{CFY2})}
Let $(\M,\tau)$ be a type ${\rm II}_1$ factor,  $X\in \M$, $X=X^*$. Then we have the following equivalent results.
\begin{enumerate}
\item There is a decomposition of the identity into $N\in \NNN$ mutually orthogonal nonzero projections $E_j$, $I=\sum_{j=1}^NE_j$, for which $E_jXE_j=\tau(X) E_j$ for all $j=1,...,N$;
\item There is a finite dimensional abelian von Neumann algebra $\B \subseteq \M$ such that
      \[E_{\B^\prime \cap \M}(X)=\tau(X)I; \]
where $E_{\B^\prime \cap \M}$ is the conditional expectation from $\M$ onto $\B^\prime \cap \M$;
\item There is a unitary operator $U \in \M$ such that
\[\frac{1}{N}\sum_{j=0}^{N-1}{U^*}^jXU^j=\tau(X)I;\]
\item There is a unitary operator $W \in \M$ with $W^N=I$ and
\[\frac{1}{N}\sum_{j=0}^{N-1}{W^*}^jXW^j=\tau(X)I.\]
\end{enumerate}
\end{Theorem}

\begin{Lemma}\label{lem:TFAE  A=N proj sums}{\rm (Proposition 1.5 in \cite{C-W} or Lemma 3.2 in \cite{CFY2})}
Let $(\M,\tau)$ be a type ${\rm II}_1$ factor, $A\in \M^+$, $\tau(A)=1$ and $N\in \NNN$. Then the following conditions are equivalent.
\begin{enumerate}
\item $A$ is the sum of $N$ nonzero projections;
\item There is a unitary operator $U \in \M$ with
\[\frac{1}{N}\sum_{j=0}^{N-1}{U^*}^jAU^j=I.\]
\end{enumerate}
\end{Lemma}

\section{Main theorem}
\begin{Theorem}\label{T:Main Theorem}
Let $(\mathcal{M},\tau)$ be a type ${\rm II}_1$ factor, $X_1,\cdots, X_n \in \mathcal{M}$. Then there is a family of finitely many mutually orthogonal non-zero projections $\{E_i \}_{i=1} ^{N}$ in $\mathcal{M}$ such that $\sum\limits_{i=1} ^{N}E_i =I$ and $E_iX_jE_i=\tau (X_j)E_i$ for $1\leq i\leq N$ and $1\leq j\leq n$. That is, $X_j$ can be expressed as follows
\[  X_j=
  \bordermatrix	
  {&E_1&E_2&\cdots&E_N \cr
  	E_1 &\tau(X_j) & (X_j)_{12} & \cdots & (X_j)_{1N} \cr
  	E_2 &(X_j)_{21}& \tau(X_j) & \cdots & (X_j)_{2N} \cr
  	\vdots & 	\vdots & \vdots & \ddots & \vdots\cr
  	E_N & (X_j)_{N1} & (X_j)_{N2} & \cdots &\tau(X_j) \cr}.
  \]
Equivalently, there is a unitary operator $U\in \M$ such that \[\frac{1}{N} \sum_{i=0}^{N-1}(U^i)^\ast X_j U^i=\tau(X_j)I\] for $1\leq j\leq n$.
\end{Theorem}

To prove Theorem~\ref{T:Main Theorem}, we may assume that $X_i=X_i^*$ and $\tau(X_i)=0$ for $1\leq i\leq n$. In the following we assume that $n=2$. The general case can be proved by mathematical induction. We assume that $H=H^*$, $K=K^*$ in $\M$ such that $\tau(H)=\tau(K)=0$. By Theorem \ref{thm: adjoint operator}, there is a family of mutually orthogonal nonzero projections $\{E_i \}_{i=1} ^{N}$ in $\mathcal{M}$ such that $\sum\limits_{i=1} ^{N}E_i =I$ and $E_iHE_i=0$ for $1\leq i\leq N$. So we have
\[
 H=
  \bordermatrix	
  {&E_1&E_2&\cdots&E_N \cr
  	E_1 &0 & \ast & \cdots & \ast \cr
  	E_2 &\ast& 0 & \cdots & \ast \cr
  	\vdots & 	\vdots & \vdots & \ddots & \vdots\cr
  	E_N & \ast & \ast & \cdots & 0 \cr},\quad   K=
  \bordermatrix	
  {&E_1&E_2&\cdots&E_N \cr
  	E_1 &K_1 & \ast & \cdots & \ast \cr
  	E_2 &\ast& K_2 & \cdots & \ast \cr
  	\vdots & 	\vdots & \vdots & \ddots & \vdots\cr
  	E_N & \ast & \ast & \cdots & K_N \cr}.
  \]

By~Theorem \ref{thm: adjoint operator} again, there is family of mutually orthogonal nonzero projections $\{F_i \}_{i=1} ^{L}$ in $\mathcal{M}$ such that $\sum\limits_{i=1} ^{L}F_i =I$ and $F_iHF_i=0$, $F_iKF_i=\alpha_iF_i$ for $1\leq i\leq L$. Now we have
\[
 H=
  \bordermatrix	
  {&F_1&F_2&\cdots&F_L \cr
  	F_1 &0 & \ast & \cdots & \ast \cr
  	F_2 &\ast& 0 & \cdots & \ast \cr
  	\vdots & 	\vdots & \vdots & \ddots & \vdots\cr
  	F_L & \ast & \ast & \cdots & 0 \cr},\quad   K=
  \bordermatrix	
  {&F_1&F_2&\cdots&F_L \cr
  	F_1 &\alpha_1 & \ast & \cdots & \ast \cr
  	F_2 &\ast& \alpha_2 & \cdots & \ast \cr
  	\vdots & 	\vdots & \vdots & \ddots & \vdots\cr
  	F_L & \ast & \ast & \cdots & \alpha_L \cr}.
  \]

By mathematical induction, we only need to show that for
\[
 H=
  \bordermatrix	
  {&F_1&F_2&\cdots&F_L \cr
  	F_1 &0 & \ast & \cdots & \ast \cr
  	F_2 &\ast& 0 & \cdots & \ast \cr
  	\vdots & 	\vdots & \vdots & \ddots & \vdots\cr
  	F_L & \ast & \ast & \cdots & 0 \cr},\quad   K=
  \bordermatrix	
  {&F_1&F_2&\cdots&F_L \cr
  	F_1 &\alpha & \ast & \cdots & \ast \cr
  	F_2 &\ast& \alpha & \cdots & \ast \cr
  	\vdots & 	\vdots & \vdots & \ddots & \vdots\cr
  	F_L & \ast & \ast & \cdots & \beta \cr}
  \]
we can rewrite $H$ and $K$ as
\[
 H=
  \bordermatrix	
  {&P_1&P_2&\cdots&P_s \cr
  	P_1 &0 & \ast & \cdots & \ast \cr
  	P_2 &\ast& 0 & \cdots & \ast \cr
  	\vdots & 	\vdots & \vdots & \ddots & \vdots\cr
  	P_s & \ast & \ast & \cdots & 0 \cr},\quad   K=
  \bordermatrix	
  {&P_1&P_2&\cdots&P_s \cr
  	P_1 &\tau(K) & \ast & \cdots & \ast \cr
  	P_2 &\ast& \tau(K) & \cdots & \ast \cr
  	\vdots & 	\vdots & \vdots & \ddots & \vdots\cr
  	P_s & \ast & \ast & \cdots & \tau(K) \cr}
  \]
for a finite decomposition $I=P_1+\cdots+P_s$. We may assume that $\tau(K)=0$, i.e., $\alpha(\tau(F_1)+\cdots+\tau(F_{L-1}))+\beta\tau(F_L)=0$. By cutting $F_L$ into suitable small pieces, i.e., $F_L=F_{L,1}+\cdots+F_{L,L-1}$, we may assume that $\alpha\tau(F_i)+\beta\tau(F_{L,i})=0$ for $1\leq i\leq L-1$. Therefore, we only need to consider the case:
\[
H=\begin{pmatrix}
0&\quad \ast\\
\ast&\quad 0
\end{pmatrix}, \quad K=\begin{pmatrix}
\alpha &\quad \ast\\
\ast&\quad \beta
\end{pmatrix},
\]
where $\tau(K)=0$.

We consider the special case $\alpha=1$ and $\beta=-1$ first. Then we can write
\[
K=
\begin{pmatrix}
1&\quad C   \\
C^*&\quad -1
\end{pmatrix},\quad
H=
\begin{pmatrix}
0& \quad D   \\
\quad D^*& 0
\end{pmatrix}.
\]
By the polar decomposition theorem, $D=U|D|$, where $U$ is a unitary operator. Thus
\[
\begin{pmatrix}
0& D   \\
D^*& 0
\end{pmatrix}
=\begin{pmatrix}
U& 0   \\
0& I
\end{pmatrix}
\begin{pmatrix}
0& |D|   \\
|D|& 0
\end{pmatrix}
\begin{pmatrix}
U^*& 0   \\
0& I
\end{pmatrix}
\]
\[
\begin{pmatrix}
U^*& 0   \\
0& I
\end{pmatrix}
\begin{pmatrix}
1& C   \\
C^*& -1
\end{pmatrix}
\begin{pmatrix}
U& 0   \\
0& I
\end{pmatrix}
=
\begin{pmatrix}
1& U^*C   \\
C^*U& -1
\end{pmatrix}.
\]
So we may assume that $D\geq 0$. Then there is a maximal abelian von Neumann algebra $\A$ of $\M$ containing $D$. Choose projections $P,Q \in \A$ with $P+Q=I$ and $\tau(P)=\tau(Q)=\frac{1}{2}$. Therefore, we can write
\[
H=
\begin{pmatrix}
0 & 0&  PD&0\\
0 & 0&  0&QD\\
PD & 0&  0&0\\
0 & QD&  0&0
\end{pmatrix}
\text{ and }
K=
\begin{pmatrix}
P & 0&  *&*\\
0 & Q&  *&*\\
* & *&  -P&0\\
* & *&  0&-Q
\end{pmatrix}.
\]
Note that we can rewrite $H$ and $K$ in the following way (interchange the second row and the fourth row and interchange the second column and fourth column of $H$ and $K$, respectively).
\[
H=
\begin{pmatrix}
0 & 0&  PD&0\\
0 & 0&  0&QD\\
PD & 0&  0&0\\
0 & QD&  0&0
\end{pmatrix}
\text{ and }
K=
\begin{pmatrix}
P & *&  *&0\\
* & -Q&  0&*\\
* & 0&  -P&*\\
0 & *&  *& Q
\end{pmatrix}.
\]
Since $\tau(P)=\tau(Q)$, by Theorem~\ref{thm: adjoint operator}, there are finitely many projections $\tilde{E_i}$ such that
\[
\tilde{E_i}\begin{pmatrix}
P & *& \\
* & -Q&
\end{pmatrix} \tilde{E_i}=0
\]
and finitely many projections $\tilde{F_j}$ such that
\[
\tilde{F_j}\begin{pmatrix}
-P & *& \\
* & Q&
\end{pmatrix} \tilde{F_j}=0
\]
Now the $\{\tilde{E_i},\tilde{F_j}\}$ satisfy $\sum \tilde{E_i}+\sum \tilde{F_j}=I$, $\tilde{E_i}H\tilde{E_i}=\tilde{E_i}K\tilde{E_i}=0$ and $\tilde{F_j}H\tilde{F_j}=\tilde{F_j}K\tilde{F_j}=0$.

Now, we consider the general case:
\[
H=\begin{pmatrix}
0&\quad A\\
A^*&\quad 0
\end{pmatrix}\begin{matrix}
P\\
Q
\end{matrix}
, \quad K=\begin{pmatrix}
\alpha &\quad \ast\\
\ast&\quad \beta
\end{pmatrix}\begin{matrix}
P\\
Q
\end{matrix},
\]
where $\tau(K)=\alpha\tau(P)+\beta\tau(Q)=0$.

\begin{Lemma}
Let $A=PAQ$ as above. Then there exist projections $P_1,P_2$ and $Q_1,Q_2$ such that $P_1+P_2=P$, $\tau(P_1)=\tau(P_2)=\frac{1}{2}\tau(P)$, $Q_1+Q_2=Q$, $\tau(Q_1)=\tau(Q_2)=\frac{1}{2}\tau(Q)$ and $P_1AQ_2=P_2AQ_1=0$.
\end{Lemma}
\begin{proof}
Let $A=V|A|$ be the polar decomposition of $A$. Let $\tilde{Q}=R(A^*)\leq Q$ and let $\tilde{P}=R(A)\leq P$. Then $V:\tilde{Q}\rightarrow \tilde {P}$ is a unitary operator. Note that $\tilde{Q}=R(A^*)=R(A^*A)=R(|A|)$. So we have $\tilde{Q}|A|=|A|=|A|\tilde{Q}$. Let $\A$ be a maximal abelian von Neumann algebra containing $W^*(|A|, \tilde{Q})$. Then there exist two projections $\tilde{Q}_1,\tilde{Q}_2$ in $\A$ such that $\tilde{Q}_1+\tilde{Q}_2=\tilde{Q}$ and $\tau(\tilde{Q}_1)=\tau(\tilde{Q}_2)=\frac{1}{2}\tau(\tilde{Q})$. Note that $\tilde{Q}_1|A|=|A|\tilde{Q}_1$ and $\tilde{Q}_2|A|=|A|\tilde{Q}_2$. Let $\tilde{P}_1=V\tilde{Q}_1V^*$, $\tilde{P}_2=V\tilde{Q}_2V^*$. Then $\tilde{P}_1+\tilde{P}_2=\tilde{P}$ and $\tau(\tilde{P}_1)=\tau(\tilde{P}_2)=\frac{1}{2}\tau(\tilde{P})$. Note that
\[
\tilde{P}_2A\tilde{Q}_1=V\tilde{Q}_2V^*V|A|\tilde{Q}_1=V\tilde{Q}_2|A|\tilde{Q}_1=V\tilde{Q}_2\tilde{Q}_1|A|=0,
\]
\[
\tilde{P}_1A\tilde{Q}_2=V\tilde{Q}_1V^*V|A|\tilde{Q}_2=V\tilde{Q}_1|A|\tilde{Q}_2=V\tilde{Q}_1\tilde{Q}_2|A|=0.
\]
Let $Q_1',Q_2'$ be projections such that $Q_1'+Q_2'=Q-\tilde{Q}$ and $\tau(Q_1')=\tau(Q_2')=\frac{1}{2}\tau(Q-\tilde{Q})$.
Let $P_1',P_2'$ be projections such that $P_1'+P_2'=P-\tilde{P}$ and $\tau(P_1')=\tau(P_2')=\frac{1}{2}\tau(P-\tilde{P})$.
Let $Q_1=Q_1'+\tilde{Q}_1$, $Q_2=Q_2'+\tilde{Q}_2$, $P_1=P_1'+\tilde{P}_1$, $P_2=P_2'+\tilde{P}_2$. Then $P_1,P_2,Q_1,Q_2$ satisfy the lemma.
\end{proof}

Now we come back to the proof of the general case. Write
\[
H=\begin{pmatrix}
0&0&\ast&0\\
0&0&0&\ast\\
\ast&0&0&0\\
0&\ast&0&0
\end{pmatrix}
\begin{matrix}
P_1\\
P_2\\
Q_1\\
Q_2
\end{matrix}=\begin{pmatrix}
0&0&\ast&\ast\\
0&0&\ast&\ast\\
\ast&\ast&0&0\\
\ast&\ast&0&0
\end{pmatrix}
\begin{matrix}
P_1\\
Q_2\\
P_2\\
Q_1\end{matrix},
\]
\[
K=\begin{pmatrix}
\alpha&0&\ast&\ast\\
0&\alpha&\ast&\ast\\
\ast&\ast&\beta&0\\
\ast&\ast&0&\beta
\end{pmatrix}
\begin{matrix}
P_1\\
P_2\\
Q_1\\
Q_2\end{matrix}=\begin{pmatrix}
\alpha&\ast&\ast&\ast\\
\ast&\beta&\ast&\ast\\
\ast&\ast&\alpha&\ast\\
\ast&\ast&\ast&\beta
\end{pmatrix}
\begin{matrix}
P_1\\
Q_2\\
P_2\\
Q_1\end{matrix}.
\]

Since $\alpha\tau(P_1)+\beta\tau(Q_2)=\alpha\tau(P_2)+\beta\tau (Q_1)=0$, by Theorem~\ref{thm: adjoint operator}, there are finitely many projections $E_i$ such that $E_1+\cdots+E_N=I$ and $E_iHE_i=E_iKE_i=0$ for $1\leq i\leq N$.

\section{Applications}
The following theorem answers affirmatively Question 1.1 in~\cite{DS}.
\begin{Theorem}\label{commutator}
If $(\M,\tau)$ is a type ${\rm II}_1$ factor and $A \in \M$ with $\tau(A)=0$, then A is a single commutator. If $A$ is a self-adjoint operator with $\tau(A)=0$, then $A$ is a self-commutator.
\end{Theorem}
\begin{proof}
By Theorem \ref{T:Main Theorem}, there is a family of finitely many mutually orthogonal non-zero projections $\{E_i \}_{i=1} ^{N}$ in $\mathcal{M}$ such that $\sum\limits_{i=1} ^{N}E_i =I$ and $E_iAE_i=0$ for $1\leq i\leq N$. That is, $A$ can be written as follows
\[  A=
  \bordermatrix	
  {&E_1&E_2&\cdots&E_N \cr
  	E_1 &0 & A_{12} & \cdots & A_{1N} \cr
  	E_2 &A_{21}& 0 & \cdots & A_{2N} \cr
  	\vdots & 	\vdots & \vdots & \ddots & \vdots\cr
  	E_N & A_{N1} & A_{N2} & \cdots &0 \cr}.
  \]
Then $A=BC-CB$, where
\[  B=
\begin{pmatrix}
0 & B_{12} & \cdots & B_{1N}\\
B_{21}& 0 & \cdots & B_{2N} \\
 	\vdots & \vdots & \ddots & \vdots \\
 B_{N1} & B_{N2} & \cdots &0
\end{pmatrix},\quad
C=
\begin{pmatrix}	
1 & 0 & \cdots & 0 \\
0& e^{\frac{2\pi i}{N}} & \cdots & 0 \\
 	\vdots & \vdots & \ddots & \vdots \\
 0 & 0 & \cdots &e^{\frac{(N-1)2\pi i}{N}}
\end{pmatrix}
  \]
and where the $B_{kj}$ are chosen so that $(e^{\frac{(j-1)2\pi i}{N}}-e^{\frac{(k-1)2\pi i}{N}})B_{kj}=A_{kj}, 1\leq k,j \leq N $ and $k \neq j $.

Now, suppose that $A$ is a self-adjoint element with $\tau(A)=0$. By Theorem \ref{T:Main Theorem} again
\[  A=
  \bordermatrix	
  {&E_1&E_2&\cdots&E_N \cr
  	E_1 &0 & A_{12} & \cdots & A_{1N} \cr
  	E_2 &A_{12}^*& 0 & \cdots & A_{2N} \cr
  	\vdots & 	\vdots & \vdots & \ddots & \vdots\cr
  	E_N & A_{1N}^* & A_{2N}^* & \cdots &0 \cr}.
  \]
Then $A=D-CDC^*$, where
\[  D=
\begin{pmatrix}
0 & D_{12} & \cdots & D_{1N}\\
D_{12}^*& 0 & \cdots & D_{2N} \\
 	\vdots & \vdots & \ddots & \vdots \\
 D_{1N}^* & D_{2N}^* & \cdots &0
\end{pmatrix}, \quad C=
\begin{pmatrix}	
1 & 0 & \cdots & 0 \\
0& e^{\frac{2\pi i}{N}} & \cdots & 0 \\
 	\vdots & \vdots & \ddots & \vdots \\
 0 & 0 & \cdots &e^{\frac{(N-1)2\pi i}{N}}
\end{pmatrix}
\]
and the $D_{kj}$ are chosen so that $(1-e^{\frac{(k-j)2\pi i}{N}})D_{kj}=A_{kj}, 1\leq k,j \leq N$ and $k \neq j $. Since $D$ is a self-adjoint operator, then for large enough $\lambda \in \RRR$, $D+\lambda \geqslant 0$. Set
$X:=C(D+\lambda)^{\frac{1}{2}}$, we have $A=X^*X-XX^*$.
\end{proof}

The following lemma is Proposition 2.10 of \cite{KNZ}. We provide a complete proof for reader's convenience.
\begin{Lemma}\label{lemma: two proj sum}
Let $(\M,\tau)$ be a type ${\rm II}_1$ factor and $A^*=A \in \M$ with $\tau(A)=0$. If there is a unitary $U \in \M$ for which $A+U^*AU=0$, then $A$ can be expressed as $A=aP+bQ$ with projections $P,Q\in \M$ and $a,b \in \RRR$.
\end{Lemma}

\begin{proof}
Set $B:=\frac{A}{\|A\|}$, then $\tau(B+I)=1$ and $B+I \geq 0$. Since $A+U^*AU=0$, $\frac{(B+I)+U^*(B+I)U}{2}=I$. By Lemma \ref{lem:TFAE  A=N proj sums}, there exist projections $P,Q \in \M$ such that $B+I=P+Q$. Thus we have
$A=\|A\|P-\|A\|(I-Q)$.
\end{proof}

The idea for the proof of the following theorem comes from \cite{V. Rabanovich}. This result answers affirmatively Question 6(2) in \cite{GP}.

\begin{Theorem}
Let $(\M,\tau)$ be a type ${\rm II}_1$ factor and $A^*=A \in \M$, then $A$ is a linear combination of 4 projections.
\end{Theorem}

\begin{proof}
When $\tau(A) \neq 0$, by considering $\frac{A}{\tau(A)}$, we may assume that $\tau(A)=1$. The proof for the case when $\tau(A)=0$ is similar.
By considering $R_A \M R_A$, we may assume that $R_A=I$. Let $\A$ be a maximal abelian von Neumann algebra containing $A$. Since $\M$ is a type ${\rm II}_1$ factor, $\A$ is diffuse. Thus, there exist projections $E_1,E_2 \in \A$ such that $E_1+E_2=I$ and $E_1 \sim E_2$. Note that $\M \cong M_2(\CCC) \otimes E_1 \M E_1$. Denote $\N:=E_1\M E_1$. Therefore, $A-I$ can be written as
\[A-I=
\begin{pmatrix}
A_1  & 0\\
0 & A_2
\end{pmatrix} \in M_2(\CCC) \otimes \N,
\]
where $A_1,A_2$ are self-adjoint operators in $\N$.
Write
\[
A=
\begin{pmatrix}
A_1+A_2  & 0\\
0 & 0
\end{pmatrix}+
\begin{pmatrix}
-A_2  & 0\\
0 & A_2
\end{pmatrix}
+\begin{pmatrix}
I  & 0\\
0 & I
\end{pmatrix}
\]
Since $\tau(A-I)=0$, $\tau_{E_1}(A_1+A_2)=\frac{\tau(A_1+A_2)}{\tau(E_1)}=0$. By Theorem \ref{commutator}, there exists $X \in \N$ such that $A_1+A_2=X^*X-XX^*$. We can assume that $X$ is invertible and $X^*X>3I$, because of the invariance property of the commutator $X^*X-XX^*=(X^*+tI)(X+tI)-(X+tI)(X^*+tI),t \in \CCC$. Then
\[
A=\begin{pmatrix}
X^*X-XX^*  & 0\\
0 & 0
\end{pmatrix}
+
\begin{pmatrix}
-A_2  & 0\\
0 & A_2
\end{pmatrix}
+\begin{pmatrix}
I  & 0\\
0 & I
\end{pmatrix}
\]
\[
=\begin{pmatrix}
X^*X  & 0\\
0 & -XX^*
\end{pmatrix}
+
\begin{pmatrix}
-XX^*  & 0\\
0 & XX^*
\end{pmatrix}
+\begin{pmatrix}
-A_2  & 0\\
0 & A_2
\end{pmatrix}
+\begin{pmatrix}
I  & 0\\
0 & I
\end{pmatrix}
\]
\[
=\begin{pmatrix}
X^*X  & 0\\
0 & 2I-XX^*
\end{pmatrix}
+
\begin{pmatrix}
-XX^*-A_2+I  & 0\\
0 & XX^*+A_2-I
\end{pmatrix}.
\]
The last term of the above equality can be written as $\lambda P_1-\lambda P_2$ by Lemma \ref{lemma: two proj sum}, where $P_1,P_2 \in \M$ are projections and $\lambda \in \RRR$. So to prove this theorem, we only need to prove the first term of the above equality is a linear combination of 2 projections. This conclusion can be found in \cite{KN}. For the convenience of the reader, we will rewrite the proof.

Set $B:=\begin{pmatrix}
X^*X  & 0\\
0 & 2I-XX^*
\end{pmatrix}$. Since we may assume $X$ is invertible in $\N$, by the polar decomposition theorem, there is a unitary operator $U \in \N$ and an invertible positive operator $|X|$ such that $X=U|X|$. Then we can verify
\[
B-I= \begin{pmatrix}
X^*X-I  & 0\\
0 & I-XX^*
\end{pmatrix}
=\begin{pmatrix}
I  & 0\\
0 & U
\end{pmatrix}
\begin{pmatrix}
X^*X-I  & 0\\
0 & I-X^*X
\end{pmatrix}
\begin{pmatrix}
I  & 0\\
0 & U^*
\end{pmatrix}
\]
Since $X^*X>3I$, $X^*X-I$ is an invertible positive operator. Take any real $d$ such that $X^*X-I <dI$, and define $S_1$ and $S_2$ by
\[
S_1=\frac{1}{1+d}(X^*X-1+d(X^*X-1)^{-1}) \quad \text{and} \quad S_2=\frac{1}{d-1}(-X^*X+1+d(X^*X-1)^{-1}).
\]
Define $P_3$ and $P_4$ by
\[
P_3=\frac{1}{2}\begin{pmatrix}
I+S_1  & (I-S_1^2)^{\frac{1}{2}}\\
(I-S_1^2)^{\frac{1}{2}} & I-S_1
\end{pmatrix}
\quad \text{and} \quad
P_4=\frac{1}{2}\begin{pmatrix}
I+S_2  & (I-S_2^2)^{\frac{1}{2}}\\
(I-S_2^2)^{\frac{1}{2}} & I-S_2
\end{pmatrix}.
\]
It can be verified that $P_3,P_4$ are projections in $\M_2(\CCC)\otimes \N$.
Define $a=1+d$ and $b=1-d$. By calculation we can obtain that
\[
\begin{pmatrix}
X^*X-I  & 0\\
0 & I-X^*X
\end{pmatrix}
=aP_3+bP_4-I.
\]
Therefore, $B$ is a linear combination of 2 projections. The proof is completed.
\end{proof}

The following result answers affirmatively Question 1.1 in \cite{DNZ}.

\begin{Theorem}
Let $(\mathcal{M},\tau)$ be a finite factor, $X \in \mathcal{M}$. Then there exists a normal operator  $N \in \mathcal{M}$  and a nilpotent operator $K$   such that $X= N+ K$.
\end{Theorem}

\begin{proof}
We only need to consider that $(\M,\tau)$ is a type ${\rm II}_1$ factor. By Theorem \ref{T:Main Theorem}, there is a family of finitely many mutually orthogonal non-zero projections $\{E_i \}_{i=1} ^{N}$ in $\mathcal{M}$ such that $\sum\limits_{i=1} ^{N}E_i =I$ and $E_iXE_i=\tau(X)E_i$ for $1\leq i\leq N$. That is, $X$ can be written as follows
\[  X=
  \bordermatrix	
  {&E_1&E_2&\cdots&E_N \cr
  	E_1 & \tau(X) & X_{12} & \cdots & X_{1N} \cr
  	E_2 &X_{21}& \tau(X) & \cdots & X_{2N} \cr
  	\vdots & 	\vdots & \vdots & \ddots & \vdots\cr
  	E_N & X_{N1} & X_{N2} & \cdots &\tau(X) \cr}.
  \]
Set $\lambda:=\tau(X)$, then $\lambda=\exp(i \theta)|\lambda|$ for some $\theta$.

Let
\[
S = \begin{pmatrix}
|\lambda| & \exp(-i\theta)X_{12}& \cdots & \exp(-i\theta)X_{1N}\\
\exp(i\theta)X_{12}^{\ast}&|\lambda| & \cdots & \exp(-i\theta)X_{2N}\\
\vdots & 	\vdots & \ddots & \vdots\\
\exp(i\theta)X_{1N}^{\ast} &\exp(i\theta)X_{2N}^{\ast} & 	\cdots &|\lambda|
\end{pmatrix},
\]
then $S$ is self-adjoint.

Define
\[
K=X-\exp(i \theta)
\begin{pmatrix}
|\lambda| & \exp(-i\theta)X_{12}& \cdots & \exp(-i\theta)X_{1N}\\
\exp(i\theta)X_{12}^{\ast}&|\lambda| & \cdots & \exp(-i\theta)X_{2N}\\
\vdots & 	\vdots & \ddots & \vdots\\
\exp(i\theta)X_{1N}^{\ast} &\exp(i\theta)X_{2N}^{\ast} & 	\cdots &|\lambda|
\end{pmatrix}
\]
\[
=\begin{pmatrix}
0 & 0& \cdots & 0\\
X_{21}-\exp(2i\theta)X_{12}^{\ast}&0 & \cdots & 0\\
\vdots & 	\vdots & \ddots & \vdots\\
X_{N1}-\exp(2i\theta)X_{1N}^{\ast} &X_{N2}-\exp(2i\theta)X_{2N}^{\ast} & 	\cdots &0
\end{pmatrix}.
\]
Then $K$ is a  strictly lower-triangular operator and therefore nilpotent. Write $N=\exp(i\theta) S$. Then $N$ is normal and $X=N+K$.

\end{proof}

\end{CJK}

\end{document}